%% file: davies.tex
\documentclass[oneside,english,reqno]{amsart}
\usepackage{mathpazo}
\usepackage[T1]{fontenc}
\usepackage[latin9]{inputenc}
\usepackage{color}
\usepackage{babel}
\usepackage{amsthm}
\usepackage{amstext}
\usepackage{amssymb}
\usepackage[unicode=true,pdfusetitle,
 bookmarks=true,bookmarksnumbered=false,bookmarksopen=false,
 breaklinks=false,pdfborder={0 0 0},backref=false,colorlinks=true]
 {hyperref}
\hypersetup{
 pdfborderstyle={},linkcolor=black,citecolor=black,urlcolor=blue}
\usepackage{breakurl}

\makeatletter
  \theoremstyle{plain}
  \newtheorem*{thm*}{\protect\theoremname}
 \theoremstyle{definition}
 \newtheorem*{defn*}{\protect\definitionname}
\theoremstyle{plain}
\newtheorem{thm}{\protect\theoremname}
  \theoremstyle{plain}
  \newtheorem{lem}[thm]{\protect\lemmaname}
  \theoremstyle{remark}
  \newtheorem*{rem*}{\protect\remarkname}

\usepackage{calrsfs}
\usepackage{graphics}
\usepackage{color}
\usepackage{perpage}

\MakePerPage{footnote}
\gdef\SetFigFontNFSS#1#2#3#4#5{} 

\theoremstyle{remark}
\newtheorem*{qst*}{Question}

\makeatother

  \providecommand{\definitionname}{Definition}
  \providecommand{\lemmaname}{Lemma}
  \providecommand{\remarkname}{Remark}
  \providecommand{\theoremname}{Theorem}
\providecommand{\theoremname}{Theorem}

\begin{document}

\title{A graph counterexample to davies' conjecture}

\author{Gady Kozma}
\begin{abstract}
There exists a graph with two vertices $x$ and $y$ such that the
ratio of the heat kernels $p(x,x;t)/p(y,y;t)$ does not converge as
$t\to\infty$.
\end{abstract}
\maketitle
This paper is concerned with a conjecture of Brian Davies from 1997
on the heat kernel of Riemannian manifolds \cite[\S 5]{D97}. We will
not disprove the conjecture as stated, but rather transform it to
the realm of graphs using a well-known (though informal) ``dictionary''
between these two categories, and build a graph that will serve as
a counterexample. We will make some remarks on how the construction
might be carried over back to the category of manifolds in the end,
but we will not give all details. The bulk of this paper is about
graphs.

We start by describing the conjecture in its original setting. Let
$M$ be a connected Riemannian manifold, and let $p$ be the \emph{heat
kernel} associated with the \emph{Laplace-Beltrami operator} on $M$.
Then Davies conjecture states that for any $M$ and any 3 points $x,y,z\in M$
the limit
\begin{equation}
\lim_{t\to\infty}\frac{p(x,y;t)}{p(z,z;t)}\label{eq:davies}
\end{equation}
exists and is positive. Here $p(x,y;t)$ is the value of the heat
kernel at time $t$ and at points $x$ and $y$. This property is
known as the ``strong ratio limit property'' (where the ``weak''
version is an averaged result due to D\"oblin, \cite{D38}) or SRLP
for short. So Davies' conjecture is that in these settings SRLP always
holds. SRLP holds for manifolds with one end \cite{D82} and for strongly
Liouville manifolds (i.e.\ manifolds where any positive harmonic
function is constant), see \cite[corollary 2.7]{ABJ02} who also make
interesting connections between these properties and the infinite
Brownian loop.

Ratio limit properties were considered for Markov chains even earlier.
If $M$ is any Markov chain on a countable state space, then we say
that $M$ satisfies SRLP if (\ref{eq:davies}) holds for any three
states $x$, $y$ and $z$, where $p(x,y;t)$ is the probability that
the Markov chain started at $x$ will be at $y$ at time $t$. For
general Markov chains there are a few examples where SRLP does not
hold. Clearly it does not hold when the Markov chain has some kind
of periodicity. F. J. Dyson constructed an example of an aperiodic
recurrent Markov chain which does not satisfy SRLP \cite[part I, \S 10]{C60}.
That example utilizes long chains of states with only one outgoing
edge, which the walker must traverse sequentially. In particular it
is not \emph{reversible}.

Now, the Laplace-Beltrami operator is self-adjoint so a proper analog
of Davies conjecture needs to assume that the Markov chain is reversible.
Reversible Markov chains are also known as random walks on weighted
graphs. The issue of periodicity can be dealt with by looking at random
walk in continuous time or at lazy random walk. Lazy random walk is
a walk where the walker, at every step, chooses with probability $\frac{1}{2}$
to stay where it is, and with probability $\frac{1}{2}$ moves to
one of the neighbours (with probability proportional to the weights). 

The result is this paper is 
\begin{thm*}
\label{thm:graph}There exists a connected graph $G$ with bounded
weights and two vertices $x,y\in G$ such that the heat kernel of
the lazy random walk satisfies
\begin{equation}
\frac{p(x,x;t)}{p(y,y;t)}\nrightarrow\label{eq:nolim}
\end{equation}
as $t\to\infty$.
\end{thm*}
Let us remark on the ``bounded weights'' clause. When doing analogies
between manifolds and graphs, it is often assumed that the manifold
has bounded geometry and the graph has bounded weights. Davies, however,
explicitly does not make the assumption of bounded geometry. Thus
one might wonder what is exactly the graph analog. All this is mute,
of course, since the counterexample does has bounded weights (and
hence a manifold example constructed along the same line should have
bounded geometry).

It is easy to see that in these setting (reversible, irreducible,
lazy) this ratio must be bounded between two constants independent
of $n$. Hence if it does not converge then it must fluctuate between
two values. The proof constructs the graph with two halves, denoted
$H^{e}$ and $H^{o}$ ($e$ and $o$ standing for even and odd, $H$
standing for half), which are connected by one edge, $(x,y)$. On
the ``odd scales'', $H^{e}$ will ``look like $\mathbb{Z}^{22}$''
while $H^{o}$ will ``look like $\mathbb{Z}^{3}$''. This means
that to get from $x$ to $x$ (where $x$ is on the $H^{e}$ side),
the most beneficial strategy is to move to $y$ as fast as possible,
spend most of your time on the $H^{o}$ side and return to the $H^{e}$
side only in the last minute. Clearly this would mean that $p(x,x;t)$
is smaller than $p(y,y;t)$ as the random walk starting from $y$
can stay on its side at all time, not losing the constant that $x$
needs for the maneuver. On the ``even scales'' the picture is reversed
and $y$ is at a disadvantage. See figure \ref{fig:ReRo} --- drawing
in 22 dimensions might have distracted the reader, so the figure demonstrates
the construction in 1 and 2 dimensions. The smallest two braces in
the figure are the first scale, in which $H^{o}$ is really one dimensional
and $H^{e}$ is really two-dimensional. The larger two braces indicate
the second scale. This time $H^{o}$ is a network of lines so it should
be thought about as two dimensional, while $H^{e}$ is a thick column,
so it should be thought about as one dimensional. The third scale
is only hinted at in this figure, but one can imagine that $H^{o}$
now becomes a thick band, so it is again one-dimensional, while $H^{e}$
becomes a network of these thick columns and bands, so it is back
to being two dimensional.
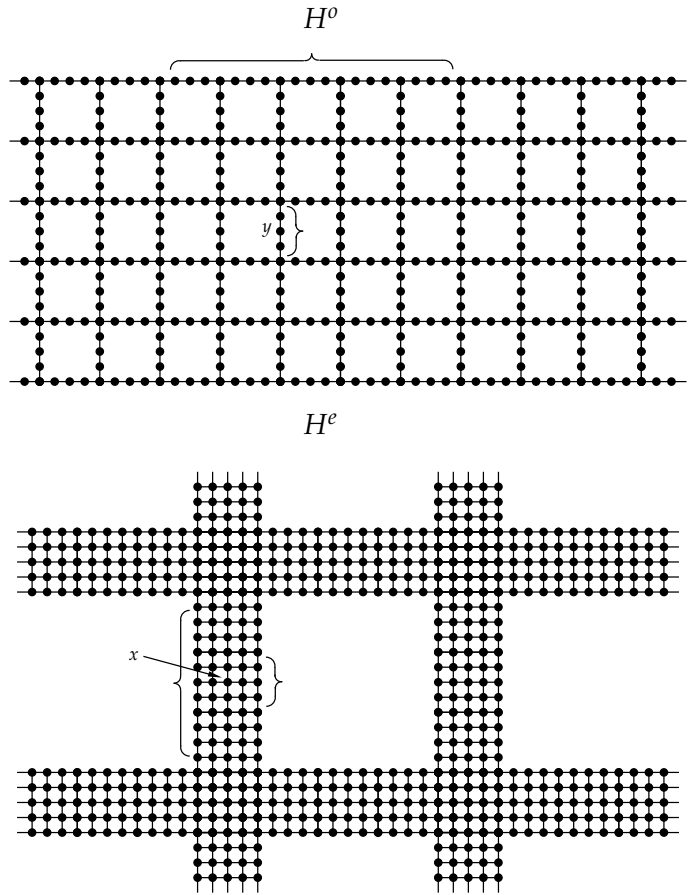
\begin{figure}
\input{12.pstex_t}

\caption{\label{fig:ReRo}The graphs $H^{o}$ and $H^{e}$}
\end{figure}

As one might expect, the numbers $3$ and $22$ have no particular
significance. They both have to be $>2$, since otherwise our graph
would be recurrent and recurrent graphs always satisfy SRLP \cite[theorem 3]{O61}.
And of course they have to be different. We took here the large value
$22$ in order to be able to be wasteful in various points (sum over
times and such stuff), but the proof could proceed with any value
larger than $3$.

This paper was first written in 2006. I wish to take this opportunity
to apologize to all those who has to wait so long for it to appear,
with no real reason. My intentions were good but my time management
was abysmal. I wish to thank Yehuda Pinchover for telling me about
the problem and for reading early drafts. Partially supported by the
Israel Science Foundation.

\section{Proof}

The construction uses $\mathbb{Z}^{d}$-like graphs as building blocks,
so we start with quoting a few results on these. We first recall the
notion of rough isometry \cite{K85}.
\begin{defn*}
Let $X$ and $Y$ be two metric spaces. We say that $X$ and $Y$
are roughly isometric if there is a constant $C$ and a map $\varphi:X\to Y$
with the following properties:
\begin{enumerate}
\item For all $x$ and $y$ in $X$, 
\[
\frac{1}{C}d(x,y)-C\le d(\varphi(x),\varphi(y))\le Cd(x,y)+C
\]

\item The image of $\varphi$ is roughly dense, i.e.\ for all $y\in Y$
there is an $x\in X$ such that $d(y,\varphi(x))\le C$
\end{enumerate}
If $G$ and $H$ are graphs we say that they are roughly isometric
if they are roughly isometric when considered with the metric $d$
being the graph distance, namely $d(x,y)$ is the length of the shortest
path between $x$ and $y$, or $\infty$ if no such path exists.
\end{defn*}
With this definition we can state the following standard result, essentially
due to Delmotte.
\begin{lem}
\label{thm:delm}Let $G$ be a graph roughly isometric to $\mathbb{Z}^{d}$.
Then the heat kernel $p$ for the lazy walk on $G$ satisfies, for
all $t\geq1$, 
\[
ct^{-d/2}\leq p(x,x;t)\leq Ct^{-d/2}.
\]

\end{lem}
Here $G$ is a simple graph --- we do not allow weights or multiple
edges. $C$ and $c$ are constant which do not depend on $t$. In
general we will use $c$ for constants which are small enough and
$C$ for constants which are large enough, and different appearance
of $c$ and $C$ might relate to different constants.
\begin{proof}
By Delmotte's theorem \cite{D99} any $G$ which satisfies volume-doubling
and the Poincar\'e inequality, satisfies 
\[
p(x,x;t)\approx\frac{1}{|B(x,\sqrt{t})|}
\]
where $B(x,r)$ is the ball around $x$ with radius $r$ (again with
the graph distance), and $|B(x,r)|$ is the sum of the degrees of
the vertices in $B$. The notation $X\approx Y$ is short for $cY\le X\le CY$.
The fact that $G$ is roughly-isometric to $\mathbb{Z}^{d}$ gives
\begin{equation}
|B(x,r)|\approx r^{d}\label{eq:Brd}
\end{equation}
so we would get $p(x,x;t)\approx t^{-d/2}$, as needed. So we need
only show that $G$ satisfies volume doubling and Poincar\'e inequality.

Now, the definition of volume doubling is that for every $x$ a vertex
of $G$ and every $r\ge1$
\[
|B(x,2r)|\le C|B(x,r)|
\]
and it follows immediately from (\ref{eq:Brd}). The Poincar\'e inequality
is not much more complicated. The definition is: for every vertex
$x$, for every $r$ and for every function $f:B(x,2r)\to\mathbb{R}$,
\begin{equation}
\sum_{y\in B(x,r)}\deg(y)|f(y)-\overline{f}|^{2}\le Cr^{2}\sum_{(y,z)\in E(B(x,2r))}(f(y)-f(z))^{2}\label{eq:Poincare}
\end{equation}
where 
\[
\overline{f}=\frac{1}{|B(x,r)|}\sum_{y\in B(x,r)}\deg(y)f(y).
\]
and where $\deg(y)$ is the degree of $y$, and $E(B)$ is the set
of edges both whose vertices are in $B$. Now, $\mathbb{Z}^{d}$ satisfies
the Poincar\'e inequality (see e.g.\ \cite[\S 4.1.1]{PSC99}). It
is well-known and not difficult to see that the Poincar\'e inequality
is preserved by rough isometries (it uses the fact that $\sum\deg(y)|f(y)-a|^{2}$
is minimized when $a=\overline{f}$). This finishes the proof.\end{proof}
\begin{lem}
\label{lem:norturnp}Let $G$ be a graph roughly-isometric to $\mathbb{Z}^{d}$,
$d\ge3$ and let $x$ be some vertex. Let $p$ be the probability
that lazy random walk starting from $x$ returns to $x$ for the first
time at time $t$. Then $p\ge ct^{-d/2}$.\end{lem}
\begin{proof}
Let $p_{1}$ be the same probability but without the restriction that
this is the first return to $x$. This is exactly $p(x,x;t)$ and
by theorem \ref{thm:delm} we have $p_{1}\ge ct^{-d/2}$. Fix some
$K$ and examine the event that the random walk returns to $x$ at
$t$ and also at some time $s\in[K,t-K]$. Let $p_{2}$ be its probability.
Using the other direction in theorem \ref{thm:delm} we can write
\[
p_{2}\le\sum_{s=K}^{t-K}p(x,x;s)p(x,x;t-s)\le C\sum_{s=K}^{t-K}s^{-d/2}(t-s)^{-d/2}\le CK^{1-d/2}t^{-d/2}.
\]
Since $d\ge3$ we can choose $K$ sufficiently large such that $p_{2}\le\frac{1}{2}p_{1}$.
So we know that with probability $p_{1}-p_{2}\ge ct^{-d/2}$ the walk
does not return to $x$ between $K$ and $t-K$. If it does reach
$x$ before time $K$, do some local modification so that it does
not. For example, if the original walker reached $x$ at some time
$s<K$ and on the next step went to some neighbour $y$ of $x$, modify
it to walk to $y$ in the first step and stay there for $s$ steps
(remember that our walk is lazy) and then continue like the original
walker. Clearly this ``costs'' only a constant and ensures our walker
does not visit $x$ in the interval $[1,K]$. Do the same for the
interval $[t-K,t-1]$, losing another constant. The finishes the lemma.\end{proof}
\begin{lem}
\label{lem:lazy}Let $G$ be a graph and let $p$ be the heat kernel
for the lazy walk on $G$. Let $t$ and $s$ satisfy that $|t-s|\leq\sqrt{t}$.
Then 
\[
|p(x,x;t)-p(x,x;s)|\leq C\frac{|t-s|\log^{3}t}{\sqrt{t}}p(x,x;t)+Ce^{-c\log^{2}t}.
\]
\end{lem}
\begin{proof}
Denote by $q(x,y;t)$ the heat kernel for the \emph{simple} random
walk on $G$. Then by definition
\[
p(x,x;t)=\sum_{i=0}^{t}q(x,x;i)\binom{t}{i}2^{-t}.
\]
Writing the same formula for $p(x,x;s)$ and subtracting we get
\[
|p(x,x;t)-p(x,x;s)|\leq\sum_{i}q(x,x;i)\left(\binom{t}{i}2^{-t}-\binom{s}{i}2^{-s}\right)=\Sigma_{1}+\Sigma_{2}
\]
where $\Sigma_{1}$ is the sum over all $|t-2i|\leq\sqrt{t}\log t$
and $\Sigma_{2}$ is the rest. A simple calculation with Stirling's
formula shows that 
\[
2^{-t}\binom{t}{i}=\sqrt{\frac{2}{\pi t}}\exp\left(-\frac{(t-2i)^{2}}{2t}\left(1+O\left(\frac{|t-2i|+1}{t}\right)\right)\right)
\]
And with some more calculations 
\[
\Sigma_{1}\leq C\sum_{|t-2i|\leq\sqrt{t}\log t}q(x,x;i)\sqrt{\frac{2}{\pi t}}e^{-(t-2i)^{2}/2t}\frac{|t-s|\log^{3}t}{\sqrt{t}}\leq C\frac{|t-s|\log^{3}t}{\sqrt{t}}p(x,x;t)
\]
while
\[
\Sigma_{2}\leq C\sum_{|t-2i|>\sqrt{t}\log t}e^{-c(t-2i)^{2}/t}\leq Ce^{-c\log^{2}t}
\]
proving the lemma.
\end{proof}

\begin{proof}
[Proof of the theorem]Abusing notations, for subsets $H\subset\mathbb{Z}^{d}$
we will not distinguish between $H$ as a set and as an induced subgraph
of $\mathbb{Z}^{d}$ ($d$ will be $22$). For the construction we
need a sufficiently fast increasing sequence $a_{1}<a_{2}<\dotsb$.
We further assume that $a_{k}$ are all even and that $a_{k-1}$ divides
$\frac{1}{2}a_{k}$. It would have probably been enough to choose
$a_{k}=2^{a_{k-1}}$, but it turns out simpler to choose the $a_{k}$
inductively, and we perform this as follows. Let $a_{1}=2$. Assume
now $a_{1},\dotsc,a_{k-1}$ have been defined. Define, for integers
$m<\frac{1}{2}l$ and $i\in\{1,\dotsc,22\}$, 
\begin{align*}
Q_{l,m,i} & :=\left\{ \vec{n}\in\mathbb{Z}^{22}:\left|n_{i}\textrm{ mod }l\right|\leq m\right\} \\
Q_{l,m} & :=\bigcup_{\substack{I\subset\{1,\dotsc,22\}\\
|I|=19
}
}\bigcap_{i\in I}Q_{l,m,i}.
\end{align*}
Here $n\textrm{ mod }l\in\{-\lfloor\frac{l-1}{2}\rfloor,\dotsc,\lfloor\frac{l}{2}\rfloor\}$.
In words, $Q_{l,m,i}$ is a $21$-dimensional subspace of $\mathbb{Z}^{22}$
orthogonal to one of the axes, fattened up by $2m+1$ (a ``slab'')
and repeated periodically with period $l$. $Q_{l,m}$ is the collection
of all $3$-dimensional subspaces, fattened and repeated similarly.
The particular point $\vec{0}$ is in fact contained in all $\binom{22}{3}$
of these $3$-dimensional slabs which will be a little inconvenient,
so let us shift $Q_{l,m}$ by 
\[
v(m)=\Big(\underbrace{{\textstyle \frac{1}{2}}m,\dotsc,{\textstyle \frac{1}{2}}m}_{3\textrm{ times}},\underbrace{\vphantom{{\textstyle \frac{1}{2}}m}0,\dotsc,0}_{19\textrm{ times}}\Big).
\]
In the shifted set $Q_{l,m}+v(m)$ the geometry of the neighbourhood
of $\vec{0}$ is simpler, it is contained in just one slab. Compare
to the figure on page \pageref{fig:ReRo}. The point $x$ is in the
\emph{middle} of a fat column and not at the intersection of a column
and a band.

We want to use these graphs with $l=a_{j}$ and $m$ a little larger
than $a_{j-1}$. Precisely, define
\[
b_{j}=\sum_{k=1}^{j}a_{k}.
\]
With this choice of $b_{j}$, $Q(a_{j},b_{j-1},i)$ contains only
complete components of $Q(a_{l},\linebreak[4]b_{l-1},i)$ for each
$l<j$. Each such component is either contained in $Q(a_{j},b_{j-1},i)$
or disjoint from it. The same holds for the translations $Q(a_{l},b_{l-1},i)+v(a_{l})$
(we need here that $a_{l}>4a_{l-1}$ so let us assume this from now
on). For brevity, define $v_{j}=v(a_{j})$.

We may now define two graphs, denoted by $H_{k-1}^{\textrm{e}}$
and $H_{k-1}^{o}$ (``e'' and ``o'' standing for even and odd)
by
\[
H_{k-1}^{\textrm{e/o}}:=\bigcap_{\substack{2\leq j\leq k-1\\
j\textrm{ even/odd}
}
}(Q_{a_{j},b_{j-1}}+v_{j}).
\]
We shall usually suppress the $k-1$ from the notation. It is not
difficult to check that $H^{\textrm{e/o}}$ are both roughly isometric
to $\mathbb{Z}^{22}$ (the rough isometry constant depends on the
``past'' $a_{1},\dotsc,a_{k-1})$. Therefore by lemma \ref{thm:delm}
we see that there exists an $\alpha$ (again, depending on the past)
such that
\begin{equation}
p_{H^{\textrm{e/o}}}(x,x;t)\leq\alpha t^{-11}.\label{eq:alpha}
\end{equation}
Examine now the graphs 
\[
F_{k-1}^{\textrm{e/o}}:=H_{k-1}^{\textrm{e/o}}\cap\left\{ \vec{n}\in\mathbb{Z}^{22}:\left|n_{i}\right|\leq b_{k-1}\:\forall i=4,\dotsc,22\right\} .
\]
$F^{\textrm{e/o}}$ are both roughly isometric to $\mathbb{Z}^{3}$
so by lemma \ref{lem:norturnp} there exists some $\beta$ such that
\begin{equation}
\mathbb{P}_{F^{\textrm{e/o}}}(\mbox{the walk returns to }\vec{0}\mbox{ for the first time at }t)\geq\frac{1}{\beta}t^{-3/2}.\label{eq:beta}
\end{equation}
 Define $\gamma_{k}:=\left\lceil \max\{\alpha,\beta\}\right\rceil $
(as usual, $\left\lceil \cdot\right\rceil $ stands for the upper
integer value). With these we can define $a_{k}$ to be any even number
satisfying $a_{k}>2\gamma_{k}^{4}+4a_{k-1}$ and such that $a_{k-1}$
divides $\frac{1}{2}a_{k}$. This completes the description of the
induction, and we define 
\[
H_{\infty}^{\textrm{e/o}}:=\bigcap_{\substack{2\leq j\\
j\textrm{ even/odd}
}
}(Q_{a_{j},a_{j-1}}+v_{j}).
\]
These graphs will be the two halves of our target graph $G$.

Before continuing, let us collect some simple facts about $H_{\infty}^{\textrm{e/o}}$:
\begin{enumerate}
\item $H_{\infty}^{\textrm{e/o}}$ is connected --- in fact we used this
indirectly when we claimed $H_{k}^{\textrm{e/o}}$ are roughly isometric
to $\mathbb{Z}^{22}$.
\item $H_{\infty}^{\textrm{e/o}}$ are transient --- this follows because
each contains a copy of $\mathbb{Z}^{3}$ (namely $\{n_{4}=\dotsb=n_{22}=0\}$)
and transience is preserved on adding edges. This last fact follows
from conductance arguments, see e.g.\ \cite{DS84}.
\end{enumerate}
Define therefore the escape probabilities
\[
\varepsilon^{\textrm{e/o}}:=\mathbb{P}_{H_{\infty}^{\textrm{e/o}}}^{\vec{0}}(R(t)\neq\vec{0}\,\forall t>0)
\]
($R$ being the random walk on the graph) and let $\delta:=\frac{1}{2}\min\{\varepsilon^{\textrm{e}},\varepsilon^{\textrm{o}}\}$.
Define the graph $G$ by connecting $H_{\infty}^{\textrm{e}}$ to
$H_{\infty}^{\textrm{o}}$ with a single edge between the two $\vec{0}$
with weight $\delta$. Define $x:=\vec{0}^{\textrm{e}}$ and $y=\vec{0}^{\textrm{o}}$.
This is our construction and we need to show (\ref{eq:nolim}), which
will follow if we show that, for $k$ sufficiently large,
\begin{equation}
\left.\begin{aligned}p(x,x;t_{2k}) & \geq3p(y,y;t_{2k})\\
p(x,x;t_{2k+1}) & \leq{\textstyle \frac{1}{3}}p(y,y;t_{2k+1})
\end{aligned}
\right\} \quad t_{k}:=\gamma_{k}^{4}.\label{eq:nk}
\end{equation}
We will only prove the even case, the odd will follow similarly.

Examine therefore $p(x,x;t_{2k})$. Since $a_{2k}>t_{2k}$ we get
that
\[
H_{\infty}^{\textrm{e/o}}\cap[-t_{2k},t_{2k}]^{22}=H_{2k}^{\textrm{e/o}}\cap[-t_{2k},t_{2k}]^{22}
\]
or in other words, the steps after $2k$ do not effect us at all.
Similarly it is possible to simplify the last stage namely
\begin{align*}
H_{2k}^{\textrm{e}}\cap[-t_{2k},t_{2k}]^{22} & =H_{2k-1}^{\textrm{e}}\cap(Q_{a_{2k},b_{2k-1}}+v_{2k})\cap[-t_{2k},t_{2k}]^{22}=\\
 & =H_{2k-1}^{\textrm{e}}\cap\left\{ \vec{n}\in\mathbb{Z}^{22}:\left|n_{i}\right|\leq b_{2k-1}\:\forall i=4,\dotsc,22\right\} =F_{2k-1}^{\textrm{e}}
\end{align*}
(here is where these translations by $v_{j}$ are used). By (\ref{eq:beta}),
\begin{align}
p_{G}(x,x;t_{2k}) & \geq\frac{1}{2}\mathbb{P}_{H_{2k}^{e}}(R\mbox{ returns to }x\mbox{ for the first time at }t)\ge\nonumber \\
 & \stackrel{\textrm{(\ref{eq:beta})}}{\ge}\frac{1}{2\gamma_{2k}}t_{2k}^{-3/2}=\frac{1}{2}t_{2k}^{-7/4}\label{eq:pxxle}
\end{align}
(the $\frac{1}{2}$ comes from the first step).

To estimate $p(y,y;t_{2k})$ we divide the event $\{R(t_{2k})=y\}$
according to whether $R$ ``essentially goes through $x$'' or not.
Formally, denote by $T_{1}$ and $T_{2}$ the first and last time
before $t_{2k}$ when $R(T)=x$ (if this does not happen, denote $T_{1}=\infty$
and $T_{2}=-\infty$). Then we define 
\begin{gather*}
p_{1}:=\mathbb{P}(T_{1}>\gamma_{2k},\, R(t_{2k})=y)\qquad p_{2}:=\mathbb{P}(T_{2}<t_{2k}-\gamma_{2k},\, R(t_{2k})=y)\\
p_{3}:=\mathbb{P}(T_{1}\leq\gamma_{2k},\, T_{2}\geq t_{2k}-\gamma_{2k},\, R(t_{2k})=y)
\end{gather*}
so that $p(y,y;t_{2k})\le p_{1}+p_{2}+p_{3}$.

Now, $p_{1}$ and $p_{2}$ are easy to estimate. As above we have
\[
H_{2k}^{\textrm{o}}\cap[-t_{2k},t_{2k}]^{22}=H_{2k-1}^{\textrm{o}}\cap[-t_{2k},t_{2k}]^{22}
\]
so (\ref{eq:alpha}) applies and we get
\begin{equation}
\mathbb{P}_{H_{\infty}^{\textrm{o}}}^{y}(R(t)=y)\leq\gamma_{2k}t^{-11}\quad\forall t\leq t_{2k}.\label{eq:yy}
\end{equation}
Therefore 
\begin{align}
p_{1} & \leq\sum_{t=\gamma_{2k}}^{t_{2k}-1}\mathbb{P}(T_{1}=t,\, R(t_{2k})=y)+\mathbb{P}(T_{1}=\infty,\, R(t_{2k})=y)\leq\nonumber \\
 & \leq\sum_{t=\gamma_{2k}-1}^{t_{2k}-2}\mathbb{P}_{H_{\infty}^{\textrm{o}}}(R(t)=y)+\mathbb{P}_{H_{\infty}^{\textrm{o}}}(R(t_{2k})=y)\leq\nonumber \\
 & \stackrel{\textrm{(\ref{eq:yy})}}{\leq}\sum_{t=\gamma_{2k}-1}^{t_{2k}-2}\gamma_{2k}\cdot t^{-11}+\gamma_{2k}\cdot t_{2k}^{-11}\leq C\gamma_{2k}^{-9}=Ct_{2k}^{-9/4}\stackrel{\textrm{(\ref{eq:pxxle})}}{=}o(p(x,x;t))\label{eq:p1}
\end{align}
and similarly for $p_{2}$. As for $p_{3}$, we have
\[
\mathbb{P}(T_{1}\leq\gamma_{2k})\leq\Big(\sum_{i=0}^{\infty}\mathbb{P}_{H_{\infty}^{\textrm{o}}}(r\textrm{ visits }y\textrm{ }i\textrm{ times before }\gamma_{2k})\Big)\cdot\delta\leq\frac{\delta}{\epsilon^{\textrm{o}}}\leq\frac{1}{2}
\]
and similarly (using time reversal) for $\mathbb{P}(T_{2}\geq t_{2k}-\gamma_{2k})$.
Hence we get
\[
p_{3}\leq\frac{1}{4}\max_{t_{2k}-2\gamma_{2k}\leq s\leq t_{2k}}p(x,x;s)
\]
and by lemma \ref{lem:lazy},
\begin{align*}
p_{3} & \leq\frac{1}{4}p(x,x;t_{2k})\left(1+O\left(\frac{\gamma_{2k}\log^{3}t_{2k}}{\sqrt{t_{2k}}}\right)\right)+O(e^{-c\log^{2}t_{2k}})\le\\
 & \stackrel{\textrm{(\ref{eq:pxxle})}}{\leq}\frac{1}{4}p(x,x;t_{2k})(1+o(1)).
\end{align*}
With the estimate (\ref{eq:p1}) for $p_{1}$ and the corresponding
estimate for $p_{2}$ we get 
\[
p(y,y;t_{2k})\leq p(x,x;t_{2k})\left(\frac{1}{4}+o(1)\right).
\]
A completely symmetric argument shows that at $t_{2k+1}$ the opposite
occurs:
\[
p(x,x;t_{2k+1})\leq p(y,y;t_{2k+1})\left(\frac{1}{4}+o(1)\right)
\]
proving the theorem.\end{proof}
\begin{rem*}
If you want an example with unweighted graphs, this is not a problem.
$H^{e}$ and $H^{o}$ are already unweighted, so the only thing needed
is to connect them, instead of with an edge of weight $\delta$, with
a segment sufficiently long such that the probability to traverse
it is $\le\delta$. The proof remains essentially the same.
\end{rem*}

\section{Manifolds}

We would like to show a Manifold $M$ and two points $x,y\in M$ such
that the heat kernel on $M$ satisfies
\[
\frac{p(x,x;t)}{p(y,y;t)}\nrightarrow
\]
as $t\to\infty$. Here is how one might translate the construction
of our theorem to the settings of manifolds. The dimension of the
manifold plays little role, so we might as well construct a surface.

For a subset $H\subset\mathbb{Z}^{22}$ one can associate a manifold
$H^{*}$ by replacing each vertex $v\in H$ with a sphere $v^{*}$
and every edge with a empty, baseless cylinder. Since the degree of
every vertex in $H$ is $\leq44$ we may simply designate 44 disjoint
circles on $\mathbb{S}^{2}$ and attach the cylinders to the spheres
at these circles.  This is reminiscent of the well-known ``infinite
jungle gym'' construction, see some lovely pictures in \cite{ON}.
The exact method of doing so is unimportant since anyway the manifold
that we get is roughly isometric to $H$, considered as an induced
subgraph of $\mathbb{Z}^{22}$ (one of the nice features of rough
isometry is that continuous and discrete objects may be roughly isometric,
rough isometry inspects only the large scale geometry). Clearly $H^{*}$
can be made $C^{\infty}$.

One can then construct a (possibly different) sequence $a_{k}$ and
two manifolds $\big(H_{\infty}^{\textrm{e/o}}\big)^{*}$ with the
only difference is that the $\alpha$ and $\beta$ must satisfy (\ref{eq:alpha})
and (\ref{eq:beta}) for our choice of the $^{*}$ operation. This
should be possible since $\big(H_{k}^{\textrm{e/o}}\big)^{*}$ and
$\big(F_{k}^{\textrm{e/o}}\big)^{*}$ are roughly isometric to $\mathbb{Z}^{22}$
and $\mathbb{Z}^{3}$ respectively. Instead of Delmotte one can use
the manifold version \cite{SC95} (or rather, \cite{D99} is the graph
version of earlier results for manifolds, see \cite{SC95} for historical
remarks).

The argument for the transience of $\big(H_{\infty}^{\textrm{e/o}}\big)^{*}$
should also be direct translation. Each contains a submanifold (with
boundary) which is roughly isometric to $\mathbb{Z}^{3}$ and therefore
is transient. Since transience is equivalent to the fact that for
some $c>0$ every function which is $1$ at $x$ and $0$ at infinity
satisfies that the Dirichlet form $\langle\nabla f,\nabla f\rangle>c$,
and since restricting to a submanifold only decreases the Dirichlet
form, we see that $\big(H_{\infty}^{\textrm{e/o}}\big)^{*}$ are transient.
Denote by 
\[
\epsilon^{\textrm{e/o}}=\inf_{x\in v^{*},v\sim\vec{0}^{\textrm{e/o}}}\mathbb{P}^{x}\left(W[0,\infty)\cap\big(\vec{0}^{\textrm{e/o}}\big)^{*}=\emptyset\right)
\]
where $W$ here is the Brownian motion on the manifold $\big(H_{\infty}^{\textrm{e/o}}\big)^{*}$;
and where the infimum is taken over all $x$ belonging to a sphere
$v^{*}$ where $v$ is some neighbor of $\vec{0}^{\textrm{e/o}}$
in $H_{\infty}^{\textrm{e/o}}$. One can now define $\delta=\frac{1}{2}\min(\varepsilon^{\textrm{e}},\varepsilon^{\textrm{o}})$
and connect $\vec{0}^{\textrm{e}}$ to $\vec{0}^{\textrm{o}}$ by
a cylinder sufficiently thin (or sufficiently long) such that the
probability to traverse it in either direction before reaching a neighboring
sphere is $\leq\delta$. This concludes a possible construction of
a manifold $M$, and one may take $x$ to be an arbitrary point in
$\big(\vec{0}^{\textrm{e}}\big)^{*}$ and $y$ and arbitrary point
in $\big(\vec{0}^{\textrm{o}}\big)^{*}$.

The proof that $p(x,x;t)/p(y,y;t)$ does not converge should not require
significant changes. We note that in our case it is possible for a
Brownian motion at time $t$ to exit the box $[-t,t]^{22}$, but it
is exponentially difficult to do so. Hence, for example, instead of
(\ref{eq:pxxle}) we get
\[
p(x,x;t_{2K})\geq\frac{1}{2\gamma_{2k}^{2}}t_{2k}^{-3/2}-Ce^{-ct_{2k}}\geq\frac{1}{4}t_{2k}^{-7/4}
\]
for $k$ sufficiently large. Another point to note is that lemma \ref{lem:lazy}
needs to be replaced with an appropriate analog.

\end{document}

%% file: 12.pstex_t
\begin{picture}(0,0)%
\includegraphics{12.pstex}%
\end{picture}%
\setlength{\unitlength}{4144sp}%
\begingroup\makeatletter\ifx\SetFigFont\undefined%
\gdef\SetFigFont#1#2#3#4#5{%
  \reset@font\fontsize{#1}{#2pt}%
  \fontfamily{#3}\fontseries{#4}\fontshape{#5}%
  \selectfont}%
\fi\endgroup%
\begin{picture}(4074,5334)(-101,-4483)
\put(623,-3076){\makebox(0,0)[lb]{\smash{{\SetFigFont{7}{8.4}{\rmdefault}{\mddefault}{\updefault}{\color[rgb]{0,0,0}$x$}%
}}}}
\put(1423,-518){\makebox(0,0)[lb]{\smash{{\SetFigFont{7}{8.4}{\rmdefault}{\mddefault}{\updefault}{\color[rgb]{0,0,0}$y$}%
}}}}
\put(1666,-1726){\makebox(0,0)[lb]{\smash{{\SetFigFont{11}{13.2}{\rmdefault}{\mddefault}{\updefault}{\color[rgb]{0,0,0}$H^e$}%
}}}}
\put(1666,704){\makebox(0,0)[lb]{\smash{{\SetFigFont{11}{13.2}{\rmdefault}{\mddefault}{\updefault}{\color[rgb]{0,0,0}$H^o$}%
}}}}
\end{picture}%